\documentclass[12pt,reqno]{amsart}

\usepackage{float}
\usepackage{mathtools}
\usepackage{times}
\usepackage[T1]{fontenc}
\usepackage{mathrsfs}
\usepackage{latexsym}
\usepackage[titletoc, title]{appendix}
\usepackage{amsmath,amsfonts,amsthm,amssymb,amscd}
\usepackage[dvipsnames]{xcolor}
\usepackage{hyperref}
\usepackage{amsmath}
\usepackage[utf8]{inputenc}

\usepackage{color}
\usepackage{breakurl}

\usepackage{comment}
\newcommand{\bburl}[1]{\textcolor{blue}{\url{#1}}}
\newcommand{\seqnum}[1]{\href{https://oeis.org/#1}{\rm \underline{#1}}}

\usepackage{caption}
\newtheorem{thm}{Theorem}[section]

\newtheorem{cor}[thm]{Corollary}

\newtheorem{lem}[thm]{Lemma}
\newtheorem{prop}[thm]{Proposition}
\newtheorem{exa}[thm]{Example}

\newtheorem{defi}[thm]{Definition}
\newtheorem{rek}[thm]{Remark}

\usepackage[utf8]{inputenc}

\numberwithin{equation}{section}

\begin{document}

\title{Schreier Sets of Intervals, Super-Schreier Sets, and Catalan Numbers}

\author[H. V. Chu]{H\`ung Vi\d{\^e}t Chu}
\email{\textcolor{blue}{\href{mailto:hchu@wlu.edu}{hchu@wlu.edu}}}
\address{Department of Mathematics, Washington and Lee University, Lexington, VA 24450, USA}  

\author[M. Khaduri]{Mariam Khaduri}
\email{\textcolor{blue}{\href{mailto:mkhaduri@mail.wlu.edu}{mkhaduri@mail.wlu.edu}}}
\address{Department of Mathematics, Washington and Lee University, Lexington, VA 24450, USA}  

\author[M. M. Khokhar]{Moiz M. Khokhar}
\email{\textcolor{blue}{\href{mailto:mkhokhar@mail.wlu.edu}{mkhokhar@mail.wlu.edu}}}
\address{Department of Mathematics, Washington and Lee University, Lexington, VA 24450, USA}

\author[R. Zhou]{Ruoan Zhou}
\email{\textcolor{blue}{\href{mailto:rzhou@mail.wlu.edu}{rzhou@mail.wlu.edu}}}
\address{Department of Mathematics, Washington and Lee University, Lexington, VA 24450, USA}

\thanks{Mariam Khaduri, Moiz M. Khokhar, and Ruoan Zhou are undergraduate students at Washington and Lee University. The authors are thankful for the support of Washington and Lee University’s 2026 Summer Research Scholars program.}



\maketitle

\begin{abstract}
A finite nonempty set $F\subset\mathbb{N}$ is Schreier if $\min F\ge |F|$.
First, we prove a linear recurrence relation and compute initial counts for Schreier sets consisting of intervals. Two intervals of integers are separated if their union is not an interval. If $\mathcal J_{k,n}$ is the collection of Schreier sets that are the union of exactly $k$ separated intervals, then the sequence $(|\mathcal{J}_{k,n}|)_{n=1}^\infty$ satisfies the characteristic polynomial  $p_k(x) = (x-1)^{2k+1}(x+1)^k$. Furthermore, we introduce the new concept of $k$-super-Schreier sets and let $\mathcal{S}_{k,n}$ denote the collection of $k$-super-Schreier sets whose maximum is $n$. We show that the sequence $(|\mathcal{S}_{k,n}|)_{n=1}^\infty$ satisfies a Fibonacci-type recurrence with a remainder term expressible as a polynomial of $n$. 
\end{abstract}
\bigskip 
\dedicatory{\small
Dedicated to Peter G. Anderson, Marjorie Bicknell-Johnson, and William A. Webb,
in recognition of their decades of dedicated service to the \textit{Fibonacci Association}
and \textit{The Fibonacci Quarterly}. Much like the recurrence relations studied
in this paper, their contributions have helped shape and sustain the Fibonacci
community for generations.}

\tableofcontents


\section{Introduction}

A finite nonempty subset $F$ of the natural numbers is said to be \textit{Schreier} if $\min F \ge \left|F\right|,$ where $\left|F\right|$ is the cardinality of $F$. Besides being useful in Banach space theory, Schreier sets are fascinating combinatorial objects that are closely related to well-known sequences and recurrences. In 2012, Bird \cite{Bi} discovered that for all $n\in \mathbb{N}$,
\[
|\{F \subset \{1,2,\ldots,n\}\,:\, F \text{ is Schreier and } n \in F\}| \ =\ F_n,
\]
where $(F_n)_{n=1}^{\infty}$ is the Fibonacci sequence with $F_1 = F_2 = 1$ and $F_n = F_{n-1} + F_{n-2}$ for $n \geq 3$.
Inspired by this unexpected observation, subsequent works have established other recurrence relations from counting sets that satisfy a variant of the Schreier condition and possess other nice properties. For example, Beanland et al.\ \cite{BCF}  discovered inclusion-exclusion type linear recurrences of higher orders by generalizing the Schreier condition to $q \min E \ge p \left|E\right|$. More recently, Beanland et al.\ \cite{BGHH} counted unions of Schreier sets and proved linear recurrences described by recursively defined characteristic polynomials. Last but not least, the Schreier-type condition $q\min F \ge \left|F\right|$ was recently used to answer an open question in approximation theory regarding the existence of a conditional $1$-suppression quasi-greedy basis \cite{BC}.

The study of combinatorial structures governed by Fibonacci-type and linear
recurrences has a long history in \textit{The Fibonacci Quarterly}. Bicknell-Johnson \cite{BJ} studied the inhomogeneous
Fibonacci-type recurrence $C_n=C_{n-1}+C_{n-2}+k$
and investigated its divisibility properties. In particular, the case of a
constant error term equal to $1$ also arises in Theorem \ref{m2} below, while
more generally that theorem produces Fibonacci-type recurrences with
nonhomogeneous error terms. Anderson and Bicknell-Johnson \cite{ABJ}
generalized Zeckendorf representations using higher-order linear recurrence
sequences, while Webb and Hamlin \cite{WH} studied restricted compositions and
characterized when their counting sequences satisfy linear recurrences.
These works illustrate two themes closely related to the present paper:
the appearance of recurrence sequences from combinatorial structures and the
study of Fibonacci-type recurrences with nonhomogeneous terms. Our results
provide further instances of these phenomena through the enumeration of
Schreier sets and the newly introduced super-Schreier sets.

An interval of integers is a set consisting of consecutive integers. Two intervals of integers $I_1$ and $I_2$ are \textit{separated} if either $\max I_1 + 1 < \min I_2$ or $\max I_2 + 1 < \min I_1$. In words, two intervals are separated if their union is not an interval. Let $\mathcal{J}_k$ be the collection of sets that are the union of exactly $k$ separated intervals. For each $n\in \mathbb{N}$, define
\[
\mathcal{J}_{k,n} \ :=\ \{F \subset \{1, 2, \dots, n\} : F \text{ is Schreier} \text{ and } F \in \mathcal{J}_k\}.
\]
Table \ref{Data_oneint} collects the initial terms of $(\left|\mathcal{J}_{k,n}\right|)_{n=1}^{\infty}$ for different $k$ values.

\begin{table}[H]
\centering
\begin{tabular}{ |c| c| c| c| c| c| c| c| c| c| c| c| c| c| c| c| c| c| c| c| c|c|}
\hline
$n$ &$1$& $2$ & $3$ & $4$ & $5$ & $6$ & $7$ & $8$ & $9$ & $10$ & $11$ & $12$ & $13$ & $14$ & $15$ & $16$\\
\hline
$|\mathcal{J}_{1,n}|$ & $1$ & $2$ & $4$ & $6$ & $9$ & $12$ & $16$ & $20$ & $25$ & $30$ & $36$ & $42$ & $49$ & $56$ & $64$ & $72$  \\
\hline
$|\mathcal{J}_{2,n}|$ & $0$ & $0$ & $0$ & $1$ & $3$ &$8$ & $16$ & $30$ & $50$ & $80$ & $120$ & $175$ & $245$ & $336$ & $448$ & $588$  \\
\hline
$|\mathcal{J}_{3,n}|$ & $0$ & $0$ & $0$ & $0$ & $0$ & $0$ & $1$ & $4$ & $13$ & $32$ & $71$ & $140$ & $259$ & $448$ & $742$ & $1176$ \\
\hline
$|\mathcal{J}_{4,n}|$ & $0$ & $0$ & $0$ & $0$ & $0$ & $0$ & $0$ & $0$ & $0$ & $1$ & $5$ & $19$ & $55$ & $140$ & $316$ & $660$ \\
\hline
$|\mathcal{J}_{5,n}|$ & $0$ & $0$ & $0$ & $0$ & $0$ & $0$ & $0$ & $0$ & $0$ & $0$ & $0$ & $0$ & $1$ & $6$ & $26$ & $86$ \\
\hline
\end{tabular}
\caption{The first $16$ values of $(|\mathcal{J}_{k, n}|)_{n=1}^\infty$ with $1\le k\le 5$. Note that $(|\mathcal{J}_{1,n}|)$ is \seqnum{A087811}; $(|\mathcal{J}_{2,n}|)$ is \seqnum{A002624}; $(|\mathcal{J}_{3,n}|)$ is \seqnum{A060099}; $(|\mathcal{J}_{4,n}|)$ is \seqnum{A060100}; $(|\mathcal{J}_{5,n}|)$ is \seqnum{A060101}.}
\label{Data_oneint}
\end{table}

\begin{defi}\normalfont Let $p(x) = c_rx^r + c_{r-1}x^{r-1} + \cdots + c_1x+c_0$ be a polynomial with real coefficients
$(c_i)_{i=0}^r$. A sequence $(a_n)_{n=1}^\infty$ is said to \textit{satisfy} $p(x)$ if
$$c_ra_n + c_{r-1}a_{n-1} + · · · + c_1a_{n-r+1} + c_0a_{n-r} = 0, \mbox{ for all }n \ge r + 1.$$
\end{defi}

Here is our first result. 

\begin{thm}\label{m1} Let $k\in \mathbb{N}$. We have
$$|\mathcal{J}_{k,n}|\ =\ \begin{cases}0, &\mbox{ if }n\le 3k-3;\\ 1, &\mbox{ if }n = 3k-2;\\ k+1, &\mbox{ if }n = 3k-1; \\ 
\binom{k+1}{2}+2k+1, &\mbox{ if }n = 3k;\\
\binom{k+1}{3}+2\binom{k+1}{2} + (k+1)^2, &\mbox{ if }n = 3k+1.
\end{cases}$$
The sequence $(|\mathcal{J}_{k,n}|)$ satisfies the polynomial $$p_k(x) = (x-1)^{2k+1}(x+1)^k.$$
\end{thm}

Next, we introduce the new concept of super-Schreier sets. 
Given a set $F$ of natural numbers and $k\in \mathbb{N}$, we use $\min_k F$ and $\max_k F$
to denote the $k$th smallest and the $k$th largest integer in $F$, respectively. Since Schreier sets have their size bounded by the minimum, we wish to give a stronger property of sets, whose size is bounded by $\min_r F$ for all $r$ up to a certain number $k$. One may naively require $\min_r F \ge |F| + (r-1)$ for all positive integers $r\le k$. Unfortunately, the requirement is equivalent to the classical requirement $\min F\ge |F|$ because
$$\min_r F \ \ge\ \min F + (r-1)\ \ge\ |F| + (r-1).$$
Hence, the collection of requirements
$$\min F\ \ge\ |F|, \min_2 F\ \ge\ |F| + 1, \ldots, \mbox{ and }\min_k F\ \ge\ |F|+(k-1)$$
does not strengthen the notion of Schreier sets. The reason is that each time we move to the next smallest integer, we increase the right side by only $1$. A natural modification is to increase the right side by $2$, which gives us the next definition.

\begin{defi}\normalfont
For $k\in\mathbb{N}$, a nonempty, finite set $F\subset\mathbb{N}$ is said to be \textit{$k$-super-Schreier} if $|F|\ge k$ and  
$$\min_r F \ \ge\ |F| + 2r-2, \mbox{ for all integers }r\in [1,k].$$
\end{defi}

For $k, n\in \mathbb{N}$, define
$$\mathcal{S}_{k, n}\ :=\ \{F\subset \{1, 2, \ldots, n\}\,:\, F\mbox{ is }k\mbox{-super-Schreier}\mbox{ and }n\in F\}.$$
A simple Python program by ChatGPT \cite{GPT} helps collect the data in Table \ref{skn}.

\begin{table}[H]
\centering
\begin{tabular}{ |c| c| c| c| c| c| c| c| c| c| c| c| c| c| c| c| c| c| c|}
\hline
$n$ &$1$& $2$ & $3$ & $4$ & $5$ & $6$ & $7$ & $8$ & $9$ & $10$ & $11$ & $12$ & $13$ & $14$ & $15$ & $16$ & $17$\\
\hline
$|\mathcal{S}_{1, n}|$  & $1$ & $1$ & $2$ & $3$ & $5$ & $8$ & $13$ & $21$ & $34$ & $55$ & $89$ & $144$ & $233$ & $377$ & $610$ & $987$ & $1597$  \\
\hline
$|\mathcal{S}_{2, n}|$ & $0$ & $0$ & $0$ & $2$ & $3$ & $6$ & $10$ & $17$ & $28$ & $46$ & $75$ & $122$ & $198$ & $321$ & $520$ & $842$ & $1363$  \\
\hline
$|\mathcal{S}_{3, n}|$ & $0$ & $0$ & $0$ & $0$ & $0$ & $0$ & $5$ & $9$ & $19$ & $34$ & $60$ & $102$ & $171$ & $283$ & $465$ & $760$ & $1238$  \\
\hline
$|\mathcal{S}_{4, n}|$ & $0$ & $0$ & $0$ & $0$ & $0$ & $0$ & $0$ & $0$ & $0$ & $14$ & $28$ & $62$ & $117$ & $214$ & $375$ & $643$ & $1083$  \\
\hline
$|\mathcal{S}_{5, n}|$ & $0$ & $0$ & $0$ & $0$ & $0$ & $0$ & $0$ & $0$ & $0$ &  $0$ &  $0$ &  $0$ & $42$ & $90$ & $207$ & $407$ & $768$\\
\hline
$|\mathcal{S}_{6, n}|$ & $0$ & $0$ & $0$ & $0$ & $0$ & $0$ & $0$ & $0$ & $0$ &  $0$ &  $0$ &  $0$ & $0$ & $0$ & $0$ & $132$ & $297$\\
\hline
\end{tabular}
\caption{The first $17$ values of $(|\mathcal{S}_{k, n}|)_{n=1}^\infty$ with integers $k \in [1,6]$. The sequence $(|\mathcal{S}_{1,n}|)$ is \seqnum{A000045} and $(|\mathcal{S}_{2,n}|)$ is \seqnum{A001610}. Meanwhile, $(|\mathcal{S}_{3,n}|)$ and $(|\mathcal{S}_{4,n}|)$ are recently added to the On-Line Encyclopedia of Integer Sequences as \seqnum{A400447} and \seqnum{A400449}, respectively.}
\label{skn}
\end{table}

We state our second result.

\begin{thm}\label{m2}
    For each $k\in \mathbb{N}$, we have
    $$|\mathcal{S}_{k, n}|\ =\ \begin{cases}
        0,&\mbox{ if }1\le n\le 3k-3;\\
        \frac{1}{k+1}\binom{2k}{k},&\mbox{ if }n = 3k-2;\\
        \frac{3}{2k+1}\binom{2k+1}{k+2}, &\mbox{ if }n = 3k-1.\\
    \end{cases}$$
    For $n\ge 3k$, we have 
    \begin{equation}\label{e10}|\mathcal{S}_{k, n}|-|\mathcal{S}_{k, n-1}|-|\mathcal{S}_{k, n-2}|\ =\ \begin{cases}0, &\mbox{ if }k = 1;\\ 1, &\mbox{ if }k=2;\\  n-4,&\mbox{ if }k=3;\\ \frac{1}{(k-2)!}\left(\prod_{i=k}^{2k-4}(n-i)\right)(n-3k+5), &\mbox{ if }k\ge 4.\end{cases}\end{equation}
\end{thm}

\begin{rek}\normalfont
Note that the number $\binom{2k}{k}/(k+1)$ in Theorem \ref{m2} is the familiar $k$th Catalan number.     
\end{rek}

\begin{exa}\normalfont
    When $k = 4$, Theorem \ref{m2} gives
    $$|\mathcal{S}_{4, n}| \ =\ 0,\mbox{ for all }1\le n\le 9, \quad |\mathcal{S}_{4, 10}| \ =\ \frac{1}{5}\binom{8}{4}\ =\ 14, \quad |\mathcal{S}_{4, 11}| \ =\ \frac{3}{9}\binom{9}{6}\ =\ 28,$$
    and for $n\ge 12$,
    $$|\mathcal{S}_{4, n}|-|\mathcal{S}_{4, n-1}|-|\mathcal{S}_{4, n-2}|\ =\ \frac{1}{2}\left(\prod_{i=4}^{4}(n-i)\right)(n-7)\ =\ \frac{1}{2}(n-4)(n-7).$$
\end{exa}

The paper is structured as follows. Section 2 proves Theorem \ref{m1}, establishing the recurrence for Schreier sets that consist of intervals. To keep our proof concise, we discuss a technical and repetitive proof in Appendix \ref{tech_cal}. Section \ref{k-super-sect} investigates $k$-super-Schreier sets and proves Theorem \ref{m2}.

\section{Recurrence for Schreier sets of intervals}
In this section, first we compute the initial terms of $(|\mathcal{J}_{k,n}|)_{n=1}^\infty$ and then prove a linear recurrence for the sequence. We prove the linear recurrence by induction, which requires us to relate two consecutive rows of Table \ref{Data_oneint}. Gemini \cite{Gemini} gives us the relation
\begin{equation}\label{AI_relation}|\mathcal{J}_{k,n}| \ =\ |\mathcal{J}_{k+1, n+3}| - |\mathcal{J}_{k+1, n+2}| - |\mathcal{J}_{k+1, n+1}| + |\mathcal{J}_{k+1, n}|, \mbox{ for all } k, n\in \mathbb{N}.\end{equation}
The authors then prove the relation themselves. To that end,  we introduce the collection $\mathcal{I}_{k,n}$, which is a collection of sets that have maximum $n$ and are the union of exactly $k$ separated intervals. Precisely,
$$
\mathcal{I}_{k,n} \ :=\ \{F \subset \{1, 2, \dots, n\}\,:\, F \text{ is Schreier}, n\in F, \text{ and } F \text{ is $k$-separated\}}.
$$
By definition, 
$$|\mathcal{I}_{k,n}| \ =\ |\mathcal{J}_{k,n}| - |\mathcal{J}_{k,n-1}|, \mbox{ for all }k, n\ge 1,$$
with the convention that $|\mathcal{J}_{k,0}| = 0$. 
For readers who are interested in seeing how one may discover \eqref{AI_relation}, we include an explanation in Appendix \ref{woAI_appen}.

\subsection{Initial counts}
\begin{lem}
    For $k, n\in\mathbb{N}$, we have
    $$|\mathcal{J}_{k,n}|\ =\ \begin{cases}
  0, &\mbox{ if } n\le 3k-3,\\
  1, &\mbox{ if } n = 3k-2.
\end{cases}$$
\end{lem}

\begin{proof}
For a set $F\in \mathcal{J}_{k,n}$, since $F$ is Schreier and consists of $k$ intervals, the maximum of $F$ is minimized when $F = F_k\ :=\ \{k, k+2, k+4, \ldots, 3k-4, 3k-2\}$. Therefore, $\mathcal{J}_{k,n}$ is empty if $n\le 3k-3$, while $\mathcal{J}_{k, 3k-2}$ contains only $F_k$.
\end{proof}

Given $a,b\in \mathbb{N}$, let
$$[a,b]_2\ :=\ \{a\le n\le b\,:\, n-a\mbox{ is even}\}.$$

\begin{lem}
    For $k\in \mathbb{N}$, we have $|\mathcal{J}_{k, 3k-1}| = k+1$. 
\end{lem}

\begin{proof}
Let $F\in \mathcal{J}_{k, 3k-1}$ with $\min F = j$.
For $j\ge k$, the maximum of a set $F$ in $\mathcal{J}_{k,3k-1}$ with $\min F = j$ is minimized when $F = F_j := [j, j+2k-2]_2$. 

If $j\ge k+1$, then $\max F\ge \max F_j = 3k-1$.
Since sets in $\mathcal{J}_{k,3k-1}$ can only contain integers up to $3k-1$, $F_{k+1}$ is the only set in $\mathcal{J}_{k, 3k-1}$ with minimum at least $k+1$.

If $j=k$, then $F_k = [k, 3k-2]_2$.
Since sets in $\mathcal{J}_{k, 3k-1}$ can contain integers up to $3k-1$, we can shift a larger part of $F_k$ to the right by $1$ to obtain $(k-1)$ more sets:
$$[k, k+2i]_2 \cup [k+2i+3, 3k-1]_2, \quad 0\le i\le k-2.$$

Therefore, we have
a total of $(k+1)$ sets, i.e., $|\mathcal{J}_{k, 3k-1}| = k+1$. 
\end{proof}

We use the star-and-bar lemma to compute $|\mathcal{J}_{k,3k}|$ and $|\mathcal{J}_{k,3k+1}|$. For its proof, see \cite[Lemma 2.1]{KKMW}.

\begin{lem}\label{star}
    The number of solutions to $x_1 + x_2 + \cdots + x_p = n$ with $x_i \ge c_i$ (for
some nonnegative number $c_i$) is $\binom{n-\sum_{i=1}^p c_i+p-1}{p-1}$.
\end{lem}

\begin{lem}
    For $k\in \mathbb{N}$, we have
    $$|\mathcal{J}_{k,3k}|\ =\ \binom{k+1}{2} +2k+1.$$
\end{lem}

\begin{proof}
Let $F\in \mathcal{J}_{k,3k}$ with $\min F = j$.
For $j\ge k$, the maximum of a set in $\mathcal{J}_{k, 3k}$ with $\min F = j$ is minimized 
when $F = F_j := [j, j+2k-2]_2$. 
Let $(x_i)_{1\le i\le k}$ be the length of the $k$ intervals in $F$, and let
$(y_j)_{1\le j\le k-1}$ be the $(k-1)$ gaps between consecutive intervals. 
We proceed by case analysis.

Case 1: $j\ge k+2$. Then $\max F\ge\max F_j \ge 3k$. Since sets in $\mathcal{J}_{k, 3k}$ can only contain
integers up to $3k$, the only set $F\in \mathcal{J}_{k, 3k}$ with $\min F\ge k+2$ is $F_{k+2}$.

Case 2: $j = k+1$. Then $F_{k+1} = [k+1, 3k-1]_2$. Since $\min F_{k+1} > |F_{k+1}|$, we may obtain other sets in two ways: first, we allow exactly one gap of $2$ between two integers in $F$, and second, we allow $F$ to contain two consecutive integers. For the former, we have
$$\sum_{i=1}^k x_i + \sum_{j=1}^{k-1} y_j \ =\ |\{k+1, k+2, \ldots, 3k\}|\ =\ 2k$$
with $x_1 = \cdots = x_k = 1$ and $y_j\ge 1$. Hence,
$\sum_{j=1}^{k-1} y_j = k$  with $y_j\ge 1$.
By Lemma \ref{star}, the number of sets is $\binom{k-(k-1)+(k-2)}{k-2}  = k-1$.
For the latter, we have
$$\sum_{i=1}^{k} x_i + \sum_{j=1}^{k-1} y_j  \ =\ |\{k+1, k+2, \ldots, 3k\}|\ =\ 2k$$
with $y_1 = \cdots = y_{k-1} = 1$ and $x_i\ge 1$. Hence,
$\sum_{i=1}^k x_i = k+1$ with $x_i\ge 1$.
By Lemma \ref{star}, the number of sets is $\binom{(k+1)-k+(k-1)}{k-1} = k$.
In this case, we have a total of $1+(k-1) + k = 2k$ sets. 

Case 3: $j = k$. Then $F_k = [k, 3k-2]_2$. Due to $\min F = k$, each interval in $F$ must have exactly $1$ element, but the gaps between consecutive elements in $F$ can reach $3$. Let $\ell$ be the sum of all the gaps. Then $k-1\le \ell\le k+1$.
We have
$\sum_{j=1}^{k-1} y_j = \ell$
with $y_j\ge 1$. 
By Lemma \ref{star}, the number of sets is
\begin{align*}
\sum_{\ell=k-1}^{k+1}\binom{\ell - (k-1) + k-2}{k-2} &\ =\ \sum_{\ell= k-1}^{k+1}\binom{\ell-1}{k-2}\\
&\ =\ \binom{k-2}{k-2} + \binom{k-1}{k-2} + \binom{k}{k-2}\\
&\ =\ k + \binom{k}{2}\ =\ \binom{k+1}{2}.
\end{align*}

Therefore, $|\mathcal{J}_{k, 3k}| = \binom{k+1}{2}+2k+1$,
as claimed. 
\end{proof}

The calculation of $|\mathcal{J}_{k, 3k+1}|$ is similar but a bit more technical. We thus move it to Appendix \ref{tech_cal}.

\subsection{Linear recurrence}
Before proving the linear recurrence in Theorem \ref{m1}, we verify that the recurrence holds when $k = 1$ (Lemma \ref{inductive_lem}) then establish a connection between two consecutive rows of Table \ref{Data_oneint} (Corollary \ref{inductive}), a result to be used in the inductive step of Theorem \ref{m1}.
\begin{lem}\label{formula_basecase}
For every $n \in \mathbb{N}$, we have $|\mathcal{I}_{1,n}|=\left\lceil n/2 \right\rceil$.
\end{lem}

\begin{proof}
Each set $F=\{k,\ldots,n-1,n\}\in \mathcal{I}_{1,n}$ is uniquely determined by its minimum $k$.
Since $F$ is Schreier, we have $k\ge n-k+1$, so 
$$\left\lceil\frac{n+1}{2}\right\rceil\ \le\ k\ \le\ n.$$
Hence, there are $n-\left\lceil(n+1)/2\right\rceil+1$
possible choices for $k$, meaning that 
$$|\mathcal{I}_{1,n}| \ =\ n-\left\lceil\frac{n+1}{2}\right\rceil+1.$$

It remains to verify that $$n-\left\lceil\frac{n+1}{2}\right\rceil+1\ =\ \left\lceil\frac{n}{2}\right\rceil.$$
If $n=2j,$ for some $j\in\mathbb{N}$, then $$n-\left\lceil\frac{n+1}{2}\right\rceil+1\ =\ 2j-(j+1)+1\ =\ j\ =\ \left\lceil\frac{n}{2}\right\rceil.$$
If $n=2j+1,$ for some $j\in\mathbb{N}$, then $$n-\left\lceil\frac{n+1}{2}\right\rceil+1\ =\ (2j+1)-(j+1)+1\ =\ j+1\ =\ \left\lceil\frac{n}{2}\right\rceil.$$
Hence, we have $|\mathcal{I}_{1,n}|=\left\lceil n/2\right\rceil$, as claimed. 
\end{proof}

\begin{cor}\label{basecase}
    For $n\ge 5$, we have
    $$|\mathcal{J}_{1,n}| \ =\ 2|\mathcal{J}_{1,n-1}| - 2|\mathcal{J}_{1,n-3}| + |\mathcal{J}_{1, n-4}|.$$
\end{cor}

\begin{proof} Let $n\ge 5$. We have
$$\left\lceil \frac{n}{2}\right\rceil + \left\lceil \frac{n-3}{2}\right\rceil \ =\ \left\lceil \frac{n}{2}\right\rceil  + \left\lceil \frac{n-1}{2}\right\rceil - 1\ =\ \left\lceil \frac{n-2}{2}\right\rceil  + \left\lceil \frac{n-1}{2}\right\rceil.$$
Hence, it follows from Lemma \ref{formula_basecase} that
$$|\mathcal{I}_{1,n}| + |\mathcal{I}_{1,n-3}|\ =\ |\mathcal{I}_{1,n-1}| + |\mathcal{I}_{1,n-2}|,$$
which can be rewritten as
$$(|\mathcal{J}_{1,n}|-|\mathcal{J}_{1, n-1}|) + (|\mathcal{J}_{1,n-3}| - |\mathcal{J}_{1, n-4}|)\ =\ (|\mathcal{J}_{1,n-1}|-|\mathcal{J}_{1, n-2}|) + (|\mathcal{J}_{1,n-2}|-|\mathcal{J}_{1, n-3}|).$$
Therefore,
$$|\mathcal{J}_{1,n}|\ =\ 2|\mathcal{J}_{1,n-1}|-2|\mathcal{J}_{1, n-3}| + |\mathcal{J}_{1, n-4}|.$$
\end{proof}

\begin{lem}\label{inductive_lem}
    For $k, n\in \mathbb{N}$, we have
    $$|\mathcal{J}_{k,n}| + |\mathcal{I}_{k+1,n+1}| \ =\ |\mathcal{I}_{k+1, n+3}|.$$
\end{lem}

\begin{proof}
We split $\mathcal{I}_{k+1,n+3}$ into 
$\mathcal{I}_{k+1,n+3}^{(1)} := \{F \in \mathcal{I}_{k+1,n+3} \,:\, n+2 \notin F\}$ and   $\mathcal{I}_{k+1,n+3}^{(2)} := \{F \in \mathcal{I}_{k+1,n+3} \,:\, n+2 \in F\}$.

First, define the map 
$$f: \mathcal{J}_{k,n} \rightarrow \mathcal{I}_{k+1,n+3}^{(1)}, \quad F \mapsto (F+1) \cup \{n+3\}.$$
The function $f$ is well-defined because for every $F\in \mathcal{J}_{k,n}$, 
$$\min f(F) \ =\ \min F + 1\ \ge\ |F| + 1 \ =\ |f(F)|, \max f(F) = n+3, \mbox{ and } \ n+2 \notin f(F).$$
Furthermore, appending $n+3$ to $F+1$ increases the number of intervals from $k$ to $k+1$ because $(\max F+1)+1\le n+2 < n+3$. Hence, the set $f(F)$ is in $\mathcal{I}_{k+1}$, and thus, $f$ is well-defined.

It is apparent from the definition that the function $f$ is injective. We show that $f$ is surjective. Pick a set $E \in \mathcal{I}_{k+1,n+3}^{(1)}$ and consider $F := (E \backslash \{n+3\}) -1 $. Since $n+2\notin E$, we have
$$
\max F\ =\ \max (E \ \backslash\ \{n+3\})-1\  \le\ n,$$
and
$$
\min F\ =\ \min E - 1\ \ge\ |E|-1\  =\ |F|.
$$
Moreover, as $\{n+3\}$ is an interval in $E$, the set $F$ consists of $k$ intervals. 
Hence, $F \in \mathcal{J}_{k,n}$. Since $f(F) = E$, the function $f$ is surjective. 

Next, define the map
$$g: \mathcal{I}_{k+1,n+1} \rightarrow \mathcal{I}_{k+1,n+3}^{(2)}, \quad F \mapsto (F+1) \cup \{n+3\}.$$
The function $g$ is well-defined because
$$\min g(F) \ =\ \min F + 1\ \ge\ |F| + 1\ =\ |g(F)|, \max g(F) \ =\ n+3, \mbox{ and }n+2 \in g(F).$$
Furthermore, since $n+3$ is appended to the last interval of $F+1$, the number of intervals in $g(F)$ is still $k+1$.

Injectivity of $g$ is obvious. We prove surjectivity. Choose $E \in \mathcal{I}_{k+1,n+3}^{(2)}$ and consider $F := (E\  \backslash \{n+3\})-1.$ Then
$$\min F\ =\ \min E -1 \ \ge\ |E|-1\ =\ |F|.$$
Since $E \in \mathcal{I}_{k+1}$ and $\{n+2,n+3\}$ are in the last interval of $E$, the set $F$ is still in $\mathcal{I}_{k+1}$ with $\max F = n+1$. Hence, the set $F$ is in $\mathcal{I}_{k+1,n+1}$, and $g(F) = E$, which proves surjectivity of $g$.

We have shown that 
$$|\mathcal{J}_{k,n}| \ =\ |\mathcal{I}_{k+1,n+3}^{(1)}|\quad\mbox{ and }\quad |\mathcal{I}_{k+1,n+1}| \ =\ |\mathcal{I}_{k+1,n+3}^{(2)}|,$$
so
$$|\mathcal{J}_{k,n}|  + |\mathcal{I}_{k+1,n+1}|\ =\  |\mathcal{I}_{k+1,n+3}|.$$
\end{proof}

\begin{cor}\label{inductive}
For $k, n\in \mathbb{N}$, 
$$|\mathcal{J}_{k,n}| \ =\ |\mathcal{J}_{k+1, n+3}| - |\mathcal{J}_{k+1, n+2}| - |\mathcal{J}_{k+1, n+1}| + |\mathcal{J}_{k+1, n}|.$$
\end{cor}
\begin{proof}
    By Lemma \ref{inductive_lem}, we have
    $$|\mathcal{J}_{k,n}|  \ =\ |\mathcal{I}_{k+1, n+3}|- |\mathcal{I}_{k+1,n+1}|,$$
    which gives
    $$|\mathcal{J}_{k,n}| \ =\ |\mathcal{J}_{k+1, n+3}| - |\mathcal{J}_{k+1, n+2}| - |\mathcal{J}_{k+1, n+1}| + |\mathcal{J}_{k+1, n}|.$$
\end{proof}

\begin{proof}[Proof of Theorem \ref{m1}]
We proceed by induction. The base case $k = 1$ states that 
$(|\mathcal{J}_{1,n}|)_{n=1}^\infty$ satisfies the polynomial $(x-1)^3(x+1) = x^4-2x^3+2x-1$, i.e., 
$$|\mathcal{J}_n|-2|\mathcal{J}_{n-1}|+2|\mathcal{J}_{n-3}|-|\mathcal{J}_{n-4}|\ =\ 0,\mbox{ for all }n\ge 5,$$
which is confirmed by Corollary \ref{basecase}.

Let $p_k(x):=(x-1)^{2k+1}(x+1)^k = c_{3k+1}x^{3k+1} + c_{3k}x^{3k} + \cdots + c_{1}x + c_0$. 
Inductive hypothesis: suppose that, for some $k\ge 1$, the sequence $(|\mathcal{J}_{k,n}|)_{n=1}^\infty$ satisfies $p_k(x)$. Since $(|\mathcal{J}_{k, n}|)_{n=1}^\infty$ satisfies $p_k(x)$, 
$$\sum_{i=0}^{3k+1}c_{i}|\mathcal{J}_{k, i+1+j}| \ =\ 0,\mbox{ for all }j\ge 0.$$
By Corollary \ref{inductive},
$$\sum_{i=0}^{3k+1}c_{i}(|\mathcal{J}_{k+1,i+4+j}|-|\mathcal{J}_{k+1,i+3+j}| -  |\mathcal{J}_{k+1,i+2+j}| + |\mathcal{J}_{k+1,i+1+j}|) \ =\ 0,\mbox{ for all }j\ge 0.$$
Substituting $|\mathcal{J}_{k+1,i+1+j}|$ by $x^i$, we obtain
\begin{align*}
&\sum_{i=0}^{3k+1}c_{i}(|\mathcal{J}_{k+1,i+4+j}|-|\mathcal{J}_{k+1,i+3+j}| -  |\mathcal{J}_{k+1,i+2+j}| + |\mathcal{J}_{k+1,i+1+j}|)\\
&\ =\ \sum_{i=0}^{3k+1}c_{i}(x^{i+3}-x^{i+2} -  x^{i+1} + x^i)\\
&\ =\ \sum_{i=0}^{3k+1}c_{i}x^i(x-1)^2(x+1)\\
&\ =\ (x-1)^2(x+1)\sum_{i=0}^{3k+1}c_ix^i\ =\ (x-1)^2(x+1)p_k(x)\ =\ p_{k+1}(x),
\end{align*}
as claimed.
\end{proof}

\section{Recurrence from counting $k$-super-Schreier sets}\label{k-super-sect}
In this section, we study $k$-super-Schreier sets by computing the initial terms of $(|\mathcal{S}_{k,n}|)_{n=1}^\infty$ and establishing a linear recurrence for the sequence. Similar to the proof of Theorem \ref{m1}, we prove the recurrence by induction. To do so, our first task is to relate two consecutive rows of Table \ref{skn}, which is achieved in Proposition \ref{p1}.

\subsection{Initial counts}

\begin{prop}\label{p0}
For $k\in \mathbb{N}$ and $n\le 3k-3$, we have $|\mathcal{S}_{k, n}| = 0$.
\end{prop}

\begin{proof}Note that $\{1, 2, \ldots, n\}$ contains a $k$-super-Schreier set if and only if $M := \{n-(k-1),\ldots,n-1, n\}$ is $k$-super-Schreier. This is because every $k$-super-Schreier set has at least $k$ elements, and $M$ consists of the $k$ largest integers in $\{1, 2, \ldots, n\}$. It follows that
$$n \ =\ \min_k M\ \ge\ |M| + 2k-2\ =\ 3k-2.$$
Hence, for $n\le 3k-3$, the set $\mathcal{S}_{k,n}$ is empty, i.e., $|\mathcal{S}_{k,n}| = 0$. 
\end{proof}

\begin{prop}\label{p10}
For $k\in \mathbb{N}$, we have 
$$|\mathcal{S}_{k, 3k-2}| \ =\ \frac{1}{k+1}\binom{2k}{k}.$$
\end{prop}

\begin{proof}
Every set in $\mathcal{S}_{k, 3k-2}$ must have exactly $k$ elements. Suppose, for a contradiction, that there is a set $F$ in $\mathcal{S}_{k,3k-2}$ with $|F| \ge k+1$. Then 
$$3k-3\ \ge\ \min_k F\ \ge\ |F| + 2k - 2\ \ge\ k + 1 + 2k-2\ =\ 3k - 1,$$
which is a contradiction. Hence, each set in $\mathcal{S}_{k,3k-2}$ can be written as
$\{a_1 < a_2 < \cdots < a_k = 3k-2\}$.
Let 
$$x_1 \ =\ a_1 - 1, \quad x_2 \ =\ a_2 - a_1 - 1 \quad, \ldots,\quad x_k\ =\ (3k-2) - a_{k-1} - 1.$$
Then 
$$x_1 + x_2 + \cdots + x_k\ =\ 2k-2,$$
and we want 
\begin{align*}
    &a_1\ =\ x_1 + 1\ \ge\ k \ \Longrightarrow\ x_1\ \ge\ k-1,\\
    &a_2\ =\ x_1 + x_2 + 2\ \ge\ k+2\ \Longrightarrow\ x_1 + x_2\ \ge\ k,\\
    &a_3\ =\ x_1 + x_2 + x_3 + 3\ \ge\ k+4\ \Longrightarrow\ \sum_{i=1}^3 x_i\ \ge\ k+1,\\
    &\vdots\\
    &a_{k-1}\ =\ \sum_{i=1}^{k-1}x_i + (k-1)\ \ge\ k + 2(k-1)-2\ \Longrightarrow\ \sum_{i=1}^{k-1} x_i\ \ge\ k+(k-1)-2.
\end{align*}
Let $y_1 := x_1 - (k-2)\ge 1$. Then each set $F$ in $\mathcal{S}_{k, 3k-2}$ corresponds to a solution $(y_1, x_2, \ldots, x_k)$ to the equation
$$y_1 + x_2 + \cdots + x_k\ =\ k,$$
under the condition 
$$y_1\ \ge\ 1, y_1 + x_2\ \ge\ 2, y_1 + x_2 + x_3\ \ge\ 3, \ldots, y_1 + x_2 + x_3 + \cdots + x_{k-1}\ \ge\ k-1.$$
It is well-known that the number of such solutions $(y_1, x_2, \ldots, x_k)$ is equal to the $k$th Catalan number $\binom{2k}{k}/(k+1)$. This is because the $k$th Catalan number counts the lattice paths from $(0,0)$ to $(k,k)$ that consist only of northward and eastward steps and never rise above the diagonal. If we consider $y_1$ the number of eastward steps before the first northward step, consider $x_2$ the number of eastward steps before the second northward step, and so on, then each solution $(y_1, x_2, \ldots, x_k)$ corresponds uniquely to a path. 
\end{proof}

\begin{lem}\label{l40}
For $k\in \mathbb{N}$, we have
$$|\mathcal{S}_{k, 3k-2}| + |\mathcal{S}_{k, 3k-1}| \ =\ |\mathcal{S}_{k+1, 3k+1}|.$$
\end{lem}

\begin{proof}
    We partition $\mathcal{S}_{k+1, 3k+1}$ into 
    $$\mathcal{S}'_{k+1, 3k+1}\ :=\ \{F\subset \mathcal{S}_{k+1, 3k+1}\,:\, \max_2 F < 3k\}$$
     and 
     $$\mathcal{S}''_{k+1, 3k+1}\ :=\ \{F\subset \mathcal{S}_{k+1, 3k+1}\,:\, \max_2 F = 3k\}.$$
     
     Define $f: \mathcal{S}_{k, 3k-2} \rightarrow\mathcal{S}'_{k+1, 3k+1}$ as $f(F) = (F + 1)\cup \{3k+1\}$. As shown in the proof of Proposition \ref{p10}, a set $F\in \mathcal{S}_{k, 3k-2}$ has exactly $k$ elements. Then for integers $r\in [1, k]$,
     $$\min_r f(F)\ =\ \min_r F + 1\ \ge\ |F| + 2r-2 + 1\ =\ |f(F)| + 2r -2,$$
     and
     $$\min_{k+1} f(F)\ =\ 3k+1 \ =\ |f(F)| + 2(k+1)-2.$$
     Hence, the set $f(F)$ is $(k+1)$-super-Schreier. Furthermore,
     $$\max f(F)\ =\ 3k+1 \mbox{ and }\max_2 f(F)\ =\ 3k-2 + 1\ =\ 3k-1 \ <\ 3k.$$
     These verify that $f$ is well-defined. Injectivity of $f$ is clear from definition. We show surjectivity. Pick $E \in \mathcal{S}'_{k+1, 3k+1}$ and let $F:= E\backslash \{3k+1\}-1$. Then $\max F\le \max_2 E - 1 \le 3k-2$. 
     For  integers $r\in [1, k]$,
     $$\min_r F\ =\ \min_r E -1 \ \ge\ |E|+2r-2-1\ =\ |F| + 2r-2,$$
     so $F$ is $k$-super-Schreier. We know that $\max F = 3k-2$ because $\max F \le 3k-3$ contradicts Proposition \ref{p0}. These show that $F\in \mathcal{S}_{k, 3k-2}$. Since $f(F) = E$, the function $f$ is surjective.

    Next, define $g: \mathcal{S}_{k, 3k-1}\rightarrow \mathcal{S}''_{k+1, 3k+1}$ as $g(F) = (F+1)\cup \{3k+1\}$. Similar as above, one can easily verify that $g$ is a bijection. We, therefore, omit the proof.
\end{proof}

\begin{cor}\label{i_c}
    For $k\in \mathbb{N}$, we have 
    $$|\mathcal{S}_{k, 3k-1}| = \frac{3}{2k+1}\binom{2k+1}{k+2}.$$
\end{cor}

\begin{proof}
    By Proposition \ref{p10} and Lemma \ref{l40}, we have
    \begin{align*}
        |\mathcal{S}_{k, 3k-1}|&\ =\ |\mathcal{S}_{k+1, 3k+1}|-|\mathcal{S}_{k, 3k-2}|\\
        &\ =\ \frac{1}{k+2}\binom{2k+2}{k+1}- \frac{1}{k+1}\binom{2k}{k}\\
        &\ =\ \frac{1}{k+2}\frac{(2k+2)!}{(k+1)!(k+1)!} - \frac{1}{k+1}\frac{(2k)!}{k!k!}\\
        &\ =\ \frac{2k+2}{k(k+1)}\frac{(2k+1)!}{(k+2)!(k-1)!} - \frac{k+2}{k(2k+1)}\frac{(2k+1)!}{(k+2)!(k-1)!}\\
        &\ =\ \frac{3}{2k+1}\binom{2k+1}{k+2}.
    \end{align*}
\end{proof}

\subsection{Linear recurrence}

The following key result gives a relationship between every two consecutive rows of Table \ref{skn}, which is used in the inductive step of our proof of \eqref{e10}.
\begin{prop}\label{p1}For $n\ge 3k-2$, 
     \begin{align}\label{e1}(|\mathcal{S}_{k, n+3}| - |\mathcal{S}_{k, n+2}| - |\mathcal{S}_{k, n+1}|) + &(|\mathcal{S}_{k+1, n+4}| - |\mathcal{S}_{k+1, n+3}| - |\mathcal{S}_{k+1, n+2}|)\nonumber\\
     &\ = \  |\mathcal{S}_{k+1, n+5}| - |\mathcal{S}_{k+1, n+4}| - |\mathcal{S}_{k+1, n+3}|.
     \end{align}
\end{prop}

For $k, n\in \mathbb{N}$, let
\begin{align*}\mathcal{S}^{(1)}_{k, n}&\ :=\ \{F\in \mathcal{S}_{k, n}\,:\, \mbox{either }|F| = 1\mbox{ or }\max_2 F \le n-2\},\\
\mathcal{S}^{(2)}_{k, n}&\ :=\ \{F\in \mathcal{S}_{k, n}\,:\,|F| \ge 2, \max_2 F = n-1, |F| = k\},\mbox{ and }\\
\mathcal{S}^{(3)}_{k, n}&\ :=\ \{F\in \mathcal{S}_{k, n}\,:\,|F| \ge 2, \max_2 F = n-1, |F| \ge k+1\}\\
&\ =\ \{F\in \mathcal{S}_{k, n}\,:\, \max_2 F = n-1, |F| \ge k+1\}.
\end{align*}
Then $(\mathcal{S}^{(1)}_{k,n}, \mathcal{S}^{(2)}_{k,n}, \mathcal{S}^{(3)}_{k,n})$ partitions $\mathcal{S}_{k,n}$.
Define
\begin{align*}
&\Psi_{k,n}: \mathcal{S}_{k, n-1}\rightarrow \mathcal{S}^{(1)}_{k, n}\quad \mbox{ as }\quad \Psi_{k,n}(F)\ =\ (F\backslash \{n-1\})\cup \{n\}\mbox{ and}\\
&\Phi_{k,n}:  \mathcal{S}_{k, n-2}\rightarrow \mathcal{S}^{(3)}_{k, n}\quad \mbox{ as }\quad
\Phi_{k,n}(F)\ =\ (F+1)\cup\{n\}.
\end{align*}
\begin{lem}\label{l1}
 For $k\ge 1$ and $n\ge 3k-1$, the map $\Psi_{k,n}:\mathcal{S}_{k, n-1}\rightarrow \mathcal{S}^{(1)}_{k, n}$  is a bijection.
\end{lem}
\begin{proof}
Let $F\in \mathcal{S}_{k, n-1}$. If $|F| = 1$, then $k = 1$, $F = \{n-1\}$, and $\Psi_{1, n}(F) = \{n\}$, which is in $\mathcal{S}^{(1)}_{1, n}$. If $|F|\ge 2$, then $\max_2 F\le n-2$.
    Since $\Psi_{k,n}$ increases $\max F$ by $1$, the set $\Psi_{k,n}(F)$ is still $k$-super-Schreier. Furthermore, we have $n\in \Psi_{k,n}(F)$, and $\max_2 \Psi_{k,n}(F) = \max_2 F \le n-2$. Hence, we have $\Psi_{k,n}(F)\in \mathcal{S}^{(1)}_{k, n}$, and the map $\Psi_{k,n}$ is well-defined.

    It follows immediately from the definition that $\Psi_{k,n}$ is injective. We show that
    $\Psi_{k,n}$ is surjective. Pick $E\in \mathcal{S}^{(1)}_{k, n}$. Let $F := (E\backslash \{n\})\cup \{n-1\}$. If $|E| = 1$, then $k = 1$ and $E = \{n\}$. We have $F = \{n-1\}\in \mathcal{S}_{1,n-1}$. Assume that $|E|\ge 2$ to have $\max_2 E \le n-2$. Note that $\max F = n-1$, and for every integer $r<k$, 
    $$\min_r F \ =\ \min_r E\ \ge\ |E| + 2r - 2\ =\ |F| + 2r-2.$$
    If $k < |F|$, then
    $$\min_k F\ =\ \min_k E \ \ge\ |E| + 2k-2\ =\ |F| + 2k-2.$$
    If $k = |F|$, then  
    $$\min_k F\ =\ n-1\ \ge\ 3k-2\ =\ |F| + 2k-2.$$
    Therefore, the set $F$ is $k$-super-Schreier and thus, is in $\mathcal{S}_{k, n-1}$. Since $\Psi_{k,n}(F) = E$, the map $\Psi_{k,n}$ is surjective. 
\end{proof}

\begin{lem}\label{l2}
 For $k\ge 1$ and $n\ge 3$, the map $\Phi_{k,n}:\mathcal{S}_{k, n-2}\rightarrow \mathcal{S}^{(3)}_{k, n}$  is a bijection.
\end{lem}

\begin{proof}
    Let $F\in \mathcal{S}_{k, n-2}$. Then $\max \Phi_{k,n}(F) = n$, and 
    for $r \le k$, 
    $$\min_r \Phi_{k,n}(F)\ =\ \min_r F + 1\ \ge\ |F| + 2r-2 + 1\ =\ |\Phi_{k,n}(F)| + 2r-2.$$
    Hence, the set $\Phi_{k,n}(F)$ is in $\mathcal{S}_{k,n}$. Furthermore, we have 
    $$\max_2 \Phi_{k,n}(F)\ =\ n-1\quad \mbox{ and }\quad |\Phi_{k,n}(F)| \ =\ |F| + 1\ \ge\ k + 1.$$ 
    Therefore, the set $\Phi_{k,n}(F)$ is in $\mathcal{S}^{(3)}_{k,n}$, and $\Phi_{k,n}$ is well-defined. 

    Clearly, the map $\Phi_{k,n}$ is injective by definition. We show that $\Phi_{k,n}$ is surjective. Let $E\in \mathcal{S}^{(3)}_{k, n}$ and define $F := (E\backslash \{n\})-1$. 
    Note that $|E|\ge k+1$, so for all integers $r\in [1, k]$, we have $\min_r F\ =\ \min_r E - 1$. Therefore, for 
    $1\le r \le k$,
    $$\min_r F\ =\ \min_r E - 1\ \ge\ |E| + 2r-2 - 1\ =\ |F| + 2r-2.$$
    Hence, the set $F$ is in $\mathcal{S}_{k, n-2}$.
    Since $\Phi_{k,n}(F) = E$, we have that $\Phi_{k,n}$ is surjective.
\end{proof}

\begin{cor}\label{p2}For $k\ge 1$ and $n\ge \max\{3, 3k-1\}$, we have
$$|\mathcal{S}_{k,n}| - |\mathcal{S}_{k,n-1}| - |\mathcal{S}_{k,n-2}| \ =\ |\mathcal{S}^{(2)}_{k, n}|.$$
\end{cor}
\begin{proof}
    By Lemmas \ref{l1} and \ref{l2}, we have
  $$|\mathcal{S}_{k, n}| - |\mathcal{S}_{k,n-1}| - |\mathcal{S}_{k,n-2}|\\
        \ =\ |\mathcal{S}^{(1)}_{k, n}| +|\mathcal{S}^{(2)}_{k, n}| + |\mathcal{S}^{(3)}_{k, n}|  - |\mathcal{S}_{k,n-1}| - |\mathcal{S}_{k,n-2}|\ =\ |\mathcal{S}^{(2)}_{k,n}|.$$
\end{proof}

Thanks to Corollary \ref{p2}, to prove \eqref{e1} is the same as to prove
\begin{equation}\label{e2}|\mathcal{S}^{(2)}_{k, n+3}| + |\mathcal{S}^{(2)}_{k+1, n+4}|\ =\ |\mathcal{S}^{(2)}_{k+1, n+5}|,\mbox{ for all }n\ge 3k-2.\end{equation}
Note that if $k = 1$, it follows from the definition of $\mathcal{S}^{(2)}_{k,n}$ that
$$|\mathcal{S}^{(2)}_{1, n+3}|\ =\ 0\mbox{ and }|\mathcal{S}^{(2)}_{2, n+4}| \ =\ |\mathcal{S}^{(2)}_{2, n+5}|\ =\ 1,$$
so \eqref{e2} holds for $k = 1$. Assume that $k \ge 2$ so that sets in $\mathcal{S}^{(2)}_{k+1, n+5}$ have at least $3$ elements. To prove \eqref{e2}, we partition $\mathcal{S}^{(2)}_{k+1, n+5}$ into 
$$\mathcal{S}^{(2,1)}_{k+1, n+5}\ :=\ \{F\in \mathcal{S}^{(2)}_{k+1, n+5}\,:\, \max_3 F\le n+2\}$$
and
$$\mathcal{S}^{(2,2)}_{k+1, n+5}\ :=\ \{F\in \mathcal{S}^{(2)}_{k+1, n+5}\,:\, \max_3 F = n+3\}.$$
Define the maps
\begin{align*}
    &\xi_{k,n}: \mathcal{S}^{(2)}_{k+1, n+4}\ \rightarrow\ \mathcal{S}^{(2,1)}_{k+1, n+5}\quad \mbox{ as }\quad F\ \mapsto (F\backslash \{n+3\})\cup \{n+5\}\mbox{ and }\\
    &\gamma_{k,n}: \mathcal{S}^{(2)}_{k, n+3}\ \rightarrow\ \mathcal{S}^{(2,2)}_{k+1, n+5}\quad\mbox{ as }\quad F\ \mapsto (F+1)\cup \{n+5\}.
\end{align*}
\begin{lem}\label{l3}
    The map $\xi_{k,n}: \mathcal{S}^{(2)}_{k+1, n+4} \rightarrow \mathcal{S}^{(2,1)}_{k+1, n+5}$ with
    $\xi_{k,n}(F) = (F\backslash \{n+3\})\cup \{n+5\}$ is a bijection. 
\end{lem}

\begin{proof}
Let $F\in \mathcal{S}^{(2)}_{k+1, n+4}$. Since $\xi_{k,n}$ replaces $n+3$ in $F$ with $n+5$, the set $\xi_{k,n}(F)$ is still $(k+1)$-super-Schreier, and $\max \xi_{k,n}(F) = n+5$. Furthermore, we have $|\xi_{k,n}(F)| =|F| = k+1$ and $\{n+4, n+5\}\subset \xi_{k,n}(F)$ because $\{n+3, n+4\}\subset F$. Lastly, that $n+3\notin \xi_{k,n}(F)$ implies that $\max_3 \xi_{k,n}(F)\le n+2$. We have verified that $\xi_{k,n}(F)\in \mathcal{S}^{(2,1)}_{k+1, n+5}$, so $\xi_{k,n}$ is well-defined. 

Next, the map $\xi_{k,n}$ is injective by definition. To see that $\xi_{k,n}$ is surjective, pick $E\in \mathcal{S}^{(2,1)}_{k+1, n+5}$ and let $F = (E\backslash \{n+5\})\cup \{n+3\}$. 
Note that $|F| = |E| = k+1$ and $\max F = n+4$. For $r \le |F| - 2$, 
$$\min_{r} F\ =\ \min_r E \ \ge\ |E| + 2r-2\ =\ |F| + 2r-2.$$
Since $|F| = k+1$,
$$\min_k F\ =\ n+3\ \ge\ 3k-1\ =\ |F| + 2k-2$$
and
$$\min_{k+1} F\ =\ n+4\ \ge\ 3k+1\ =\ |F| + 2(k+1)-2.$$
Hence, the set $F$ is $(k+1)$-super-Schreier.
Furthermore, 
$$\max_2 F \ =\ n+3\mbox{ and }|F| \ =\ k+1.$$
Hence, $F$ is in $\mathcal{S}^{(2)}_{k+1, n+4}$ with $\xi_{k,n}(F) = E$. Therefore, the map $\xi_{k,n}$ is surjective. 
\end{proof}

\begin{lem}\label{l4}
    The map $\gamma_{k,n}: \mathcal{S}^{(2)}_{k, n+3}\rightarrow\mathcal{S}^{(2,2)}_{k+1, n+5}$ with
    $\gamma(F) = (F+1)\cup \{n+5\}$ is a bijection. 
\end{lem}

\begin{proof}
    Let $F\in \mathcal{S}^{(2)}_{k, n+3}$. Then $|F| = k$. For integers $r\in [1, k]$, 
    $$\min_r \gamma_{k,n}(F)\ =\ \min_r F + 1 \ \ge\ |F| + 2r -2 +1\ =\ |\gamma_{k,n}(F)| + 2r-2,$$
    and
    $$\min_{k+1} \gamma_{k,n}(F)\ =\ n+5\ \ge\ 3k+1\ =\ |\gamma_{k,n}(F)|+2(k+1)-2.$$
    Hence, the set $\gamma_{k,n}(F)$ is $(k+1)$-super-Schreier. Furthermore, $\max \gamma_{k,n}(F) \ =\ n+5$, 
    $$\max_2\gamma_{k,n}(F)\ =\ n+4, |\gamma_{k,n}(F)|\ =\ k+1, \mbox{ and }\max_3\gamma_{k,n}(F)\ =\ n+3.$$
    We have verified that $\gamma_{k,n}(F)$ is in $\mathcal{S}^{(2,2)}_{k+1, n+5}$.

    It is clear from the definition that $\gamma_{k,n}$ is injective. We show that $\gamma_{k,n}$ is surjective. Pick $E\in \mathcal{S}^{(2,2)}_{k+1, n+5}$ and let $F := (E\backslash \{n+5\})-1$. Note that for each integer $r\in [1,k]$, 
    $$\min_r F\ =\ \min_r E-1\ \ge\ |E|+2r-2 -1\ =\ |F| + 2r-2,$$
    so $F$ is $k$-super-Schreier. Furthermore,
    $$\max F\ =\ n+3, \max_2 F\ =\ n+2, \mbox{ and }|F| \ =\ k.$$
    Hence, the set $F$ is in $F\in \mathcal{S}^{(2)}_{k, n+3}$ with $\gamma_{k,n}(F) = E$. Therefore, the map $\gamma_{k,n}$ is surjective.
\end{proof}

\begin{proof}[Proof of Proposition \ref{p1}]
    Proposition \ref{p1} follows from Corollary \ref{p2}, Lemma \ref{l3}, and Lemma \ref{l4}.
\end{proof}

Given a sequence $(a_n)_{n=1}^\infty$, define the sequence $(a^{*}_n)_{n=1}^\infty$ as
$$a^{*}_{n-2}\ =\ a_n - a_{n-1} - a_{n-2}, \mbox{ for }n\ge 3.$$
\begin{table}[H]
\centering
\begin{tabular}{ |c| c| c| c| c| c| c| c| c| c| c| c| c| c| c| c| c|}
\hline
$n$ &$1$& $2$ & $3$ & $4$ & $5$ & $6$ & $7$ & $8$ & $9$ & $10$ & $11$ & $12$ & $13$ & $14$ & $15$ \\
\hline
$|\mathcal{S}_{1, n}|^*$ & $0$ & $0$ & $0$ & $0$ & $0$ & $0$ & $0$ & $0$ & $0$ & $0$ & $0$ & $0$ & $0$ & $0$ & $0$  \\
\hline
$|\mathcal{S}_{2, n}|^*$ & $0$ & $2$ & $1$ & $1$ & $1$ & $1$ & $1$ & $1$ & $1$ & $1$ & $1$ & $1$ & $1$ & $1$ & $1$ \\
\hline
$|\mathcal{S}_{3, n}|^*$ & $0$ & $0$ & $0$ & $0$ & $5$ & $4$ & $5$ & $6$ & $7$ & $8$ & $9$ & $10$ & $11$ & $12$ & $13$   \\
\hline
$|\mathcal{S}_{4, n}|^*$ & $0$ & $0$ & $0$ & $0$ & $0$ & $0$ & $0$ & $14$ & $14$ & $20$ & $27$ & $35$ & $44$ & $54$ & $65$  \\
\hline
$|\mathcal{S}_{5, n}|^*$ & $0$ & $0$ & $0$ & $0$ & $0$ & $0$ & $0$ & $0$ & $0$ &  $0$ &  $42$ &  $48$ & $75$ & $110$ & $154$ \\
\hline
$|\mathcal{S}_{6, n}|^*$ & $0$ & $0$ & $0$ & $0$ & $0$ & $0$ & $0$ & $0$ & $0$ &  $0$ &  $0$ &  $0$ & $0$ & $132$ & $165$ \\
\hline
\end{tabular}
\caption{The first $15$ values of $(|\mathcal{S}_{k, n}|^*)_{n=1}^\infty$ with integers $k \in [1,6]$.}
\end{table}

Using the new notation, we restate Proposition \ref{p1} as follows.
\begin{prop}\label{c11}For $k\in\mathbb{N}$ and $n\ge 3k-2$, 
     \begin{equation*}|\mathcal{S}_{k, n+1}|^*  
     \ = \  |\mathcal{S}_{k+1, n+3}|^*-|\mathcal{S}_{k+1, n+2}|^*.
     \end{equation*}
\end{prop}

\begin{cor}\label{c10}
    For $k\in \mathbb{N}$ and $n\ge 3k-2$, 
    $$|\mathcal{S}_{k+1, n+3}|^*\ =\ |\mathcal{S}_{k+1, 3k}|^* + \sum_{i = 3k-1}^{n+1}|\mathcal{S}_{k, i}|^*.$$
    Equivalently, 
    $$|\mathcal{S}_{k+1, n-2}|^*\ =\ |\mathcal{S}_{k+1, 3k}|^* + \sum_{i = 3k-1}^{n-4}|\mathcal{S}_{k, i}|^*, \quad k\in \mathbb{N}, n\ge 3k+3.$$
\end{cor}

\begin{proof}
    Applying Proposition \ref{c11} repeatedly, we have
    \begin{align}\label{e40}
    |\mathcal{S}_{k+1, n+3}|^*&\ =\ |\mathcal{S}_{k+1, n+2}|^* + |\mathcal{S}_{k,n+1}|^*\nonumber\\
    &\ =\ |\mathcal{S}_{k+1, n+1}|^* +  |\mathcal{S}_{k,n+1}|^* + |\mathcal{S}_{k,n}|^*\nonumber\\
    &\quad \vdots\nonumber\\
    &\ =\ |\mathcal{S}_{k+1, n+3-m}|^* + |\mathcal{S}_{k, n+3-2}|^* + \cdots + |\mathcal{S}_{k, n+3-(m+1)}|^*,
    \end{align}
    as long as $n+4-m\ge 3(k+1)-2$, i.e., $m\le n-3k+3$. Plugging $m = n-3k+3$ into \eqref{e40} gives
    $$|\mathcal{S}_{k+1, n+3}|^*\ =\ |\mathcal{S}_{k+1, 3k}|^* + \sum_{i=3k-1}^{n+1}|\mathcal{S}_{k,i}|^*.$$
\end{proof}

\begin{lem}\label{l10}For $k\in\mathbb{N}$,
    $$|\mathcal{S}_{k+1, 3k}|^*\ =\ \frac{4}{k+3}\binom{2k+1}{k-1}.$$
\end{lem}

\begin{proof}
By Propositions \ref{p0} and \ref{p10} and Corollary \ref{i_c}, we have
\begin{align*}
    |\mathcal{S}_{k+1, 3k}|^*&\ =\ |\mathcal{S}_{k+1, 3k+2}| - |\mathcal{S}_{k+1, 3k+1}| - |\mathcal{S}_{k+1, 3k}|\\
    &\ =\ \frac{3}{2k+3}\binom{2k+3}{k+3} - \frac{1}{k+2}\binom{2k+2}{k+1} - 0\\
    &\ =\ \frac{3}{2k+3}\frac{(2k+3)!}{(k+3)!k!} - \frac{1}{k+2}\frac{(2k+2)!}{(k+1)!(k+1)!}\\
    &\ =\ \frac{3(2k+2)(2k+3)}{k(k+3)(2k+3)}\frac{(2k+1)!}{(k+2)!(k-1)!} - \frac{2k+2}{k(k+1)}\frac{(2k+1)!}{(k+2)!(k-1)!}\\
    &\ =\ \frac{3(2k+2)}{k(k+3)}\binom{2k+1}{k-1} - \frac{2}{k}\binom{2k+1}{k-1}\ =\ \frac{4}{k+3}\binom{2k+1}{k-1}.
\end{align*}
\end{proof}

\begin{proof}[Proof of Theorem \ref{m2}, second statement]
When $k = 1$, we have $|\mathcal{S}_{1,n}| = F_n$ for all $n\in \mathbb{N}$ according to \cite{Bi}, so 
\begin{equation}\label{e11}|\mathcal{S}_{1, n-2}|^*\ =\ |\mathcal{S}_{1, n}| - |\mathcal{S}_{1, n-1}| - |\mathcal{S}_{1,n-2}| \ =\ 0, \mbox{ for all }n\ge 3.\end{equation}

When $k = 2$, Corollary \ref{c10} and \eqref{e11} give that
\begin{equation}\label{e12}|\mathcal{S}_{2, n}| - |\mathcal{S}_{2,n-1}| - |\mathcal{S}_{2,n-2}| \ =\ |\mathcal{S}_{2, n-2}|^*\ =\ |\mathcal{S}_{2, 3}|^* + \sum_{i = 2}^{n-4}|\mathcal{S}_{1,i}|^*\ =\ 1, \mbox{ for all }n\ge 6.\end{equation}

When $k = 3$, Corollary \ref{c10} and \eqref{e12} give that 
\begin{align}\label{e13}|\mathcal{S}_{3, n}| - |\mathcal{S}_{3, n-1}| - |\mathcal{S}_{3, n-2}|\ =\ |\mathcal{S}_{3,n-2}|^*&\ =\ |\mathcal{S}_{3, 6}|^* + \sum_{i=5}^{n-4}|\mathcal{S}_{2, i}|^*\nonumber\\
&\ =\ 4 + \sum_{i=5}^{n-4}1\ =\ n-4, \mbox{ for all }n\ge 9.
\end{align}

Next, we prove by induction on $k$ that for every $k\ge 4$,
\begin{equation}\label{e14}
|\mathcal{S}_{k, n-2}|^*\ =\ \frac{1}{(k-2)!}\left(\prod_{i=k}^{2k-4}(n-i)\right)(n-3k+5), \mbox{ for all }n\ge 3k.
\end{equation}
For the base case $k = 4$, Corollary \ref{c10} and \eqref{e13} give  
$$
    |\mathcal{S}_{4, n-2}|^*\ =\ |\mathcal{S}_{4, 9}|^* + \sum_{i = 8}^{n-4}|\mathcal{S}_{3, i}|^*\ =\ 14 + \sum_{i=8}^{n-4} (i-2)\ =\ \frac{1}{2}(n-4)(n-7), \mbox{ for all }n\ge 12,
$$
so \eqref{e14} holds for $k = 4$. Assume that \eqref{e14} is true for some $k\ge 4$ and for all $n\ge 3k$.
By Corollary \ref{c10} and Lemma \ref{l10}, for $n\ge 3k+3$, we have
    \begin{align*}
    |\mathcal{S}_{k+1, n-2}|^*&\ =\ |\mathcal{S}_{k+1, 3k}|^* + \sum_{j = 3k-1}^{n-4}|\mathcal{S}_{k, j}|^*\\
    &\ =\ \frac{4}{k+3}\binom{2k+1}{k-1} + \frac{1}{(k-2)!}\sum_{j = 3k-1}^{n-4}\left(\prod_{i=k}^{2k-4}(j+2-i)\right) (j-3k+7).
    \end{align*}
    We wish to show that for all $n\ge 3k+3$,
    \begin{align}\label{e3}\frac{4}{k+3}\binom{2k+1}{k-1} + \frac{1}{(k-2)!}\sum_{j = 3k-1}^{n-4}&\left(\prod_{i=k-2}^{2k-6}(j-i)\right) (j-3k+7)\nonumber\\
    &\ =\ \frac{1}{(k-1)!}\left(\prod_{i=k+1}^{2k-2}(n-i)\right) (n-3k+2).\end{align}
    We proceed by induction on $n$. 
    For the base case $n = 3k+3$, the left-hand side of \eqref{e3} is
    \begin{align*}
        &\frac{4}{k+3}\binom{2k+1}{k-1} + \frac{1}{(k-2)!}\sum_{j = 3k-1}^{3k-1}\left(\prod_{i=k-2}^{2k-6}(j-i)\right) (j-3k+7)\\
        &\ =\ \frac{4}{k+3}\binom{2k+1}{k-1} + \frac{1}{(k-2)!}\left(\prod_{i=k-2}^{2k-6}(3k-1-i)\right) (3k-1-3k+7)\\
        &\ =\ \frac{4}{k+3}\binom{2k+1}{k-1} + \frac{6}{(k-2)!}\prod_{i=k-1}^{2k-5}(3k-i).
    \end{align*}
    Meanwhile, the right-hand side of \eqref{e3} is
    \begin{align*}\frac{1}{(k-1)!}\left(\prod_{i=k+1}^{2k-2}(3k+3-i)\right) (3k+3-3k+2)&\ =\ \frac{5}{(k-1)!}\prod_{i=k-2}^{2k-5}(3k-i)\\
    &\ =\ \frac{5}{(k-1)!}(2k+2)\prod_{i=k-1}^{2k-5}(3k-i).
    \end{align*}
    Hence, when $n = 3k+3$, \eqref{e3} states that 
    $$\frac{4}{k+3}\binom{2k+1}{k-1} + \frac{6}{(k-2)!}\prod_{i=k-1}^{2k-5}(3k-i)\ =\  \frac{5}{(k-1)!}(2k+2)\prod_{i=k-1}^{2k-5}(3k-i),$$
    which is equivalent to 
      $$\frac{4}{k+3}\binom{2k+1}{k-1} \ =\ \left(\frac{5}{(k-1)!}(2k+2)-\frac{6(k-1)}{(k-1)!}\right)\prod_{i=k-1}^{2k-5}(3k-i),$$
      i.e., 
       \begin{equation}\label{e15}\frac{4}{k+3}\binom{2k+1}{k-1} \ =\ \frac{4k+16}{(k-1)!}\prod_{i=k-1}^{2k-5}(3k-i).\end{equation}
    To see why \eqref{e15} is true,  we note that
    \begin{align*}
        \frac{4k+16}{(k-1)!}\prod_{i=k-1}^{2k-5}(3k-i)\ =\ \frac{4(k+4)}{(k-1)!}\prod_{i=k+5}^{2k+1}i&\ =\ \frac{4(k+4)}{(k-1)!}\frac{(2k+1)!}{(k+4)!}\\
        &\ =\ \frac{4}{k+3}\frac{(2k+1)!}{(k-1)!(k+2)!}\\
        &\ =\ \frac{4}{k+3}\binom{2k+1}{k-1}.
    \end{align*}

    Assume that \eqref{e3} holds for some $n\ge 3k+3$, i.e.,
      \begin{align}\label{e20}\frac{4}{k+3}\binom{2k+1}{k-1} + \frac{1}{(k-2)!}\sum_{j = 3k-1}^{n-4}&\left(\prod_{i=k-2}^{2k-6}(j-i)\right) (j-3k+7)\nonumber\\
    &\ =\ \frac{1}{(k-1)!}\left(\prod_{i=k+1}^{2k-2}(n-i)\right) (n-3k+2).\end{align}
    We need to show that 
            \begin{align}\label{e21}\frac{4}{k+3}\binom{2k+1}{k-1} + \frac{1}{(k-2)!}\sum_{j = 3k-1}^{n-3}&\left(\prod_{i=k-2}^{2k-6}(j-i)\right) (j-3k+7)\nonumber\\
            &\ =\ \frac{1}{(k-1)!}\left(\prod_{i=k}^{2k-3}(n-i)\right) (n-3k+3).\end{align}
    Subtracting \eqref{e20} from \eqref{e21} side by side, we obtain
\begin{align*}&\frac{1}{(k-2)!}\left(\prod_{i=k-2}^{2k-6}(n-3-i)\right)(n-3-3k+7)\\
&\ =\ \frac{1}{(k-1)!}\left(\prod_{i=k+1}^{2k-3}(n-i)\right)((n-k)(n-3k+3)-(n-2k+2)(n-3k+2)),\end{align*}
which is equivalent to
\begin{align}\label{e30}&(k-1)(n-3k+4)\prod_{i=k+1}^{2k-3}(n-i)\nonumber\\
&\ =\ \left(\prod_{i=k+1}^{2k-3}(n-i)\right)((n-k)(n-3k+3)-(n-2k+2)(n-3k+2)).\end{align}
Since $(k-1)(n-3k+4) = (n-k)(n-3k+3)-(n-2k+2)(n-3k+2)$, \eqref{e30} is true.
\end{proof}

\appendix
\section{Computing $|\mathcal{J}_{k, 3k+1}|$}\label{tech_cal}
\begin{lem}
    For $k\in \mathbb{N}$, we have
    $$|\mathcal{J}_{k,3k+1}|\ =\ \binom{k+1}{3}+2\binom{k+1}{2} + (k+1)^2.$$
\end{lem}

\begin{proof}
Let $F\in \mathcal{J}_{k,3k+1}$ with $\min F = j$.
For $j\ge k$, the maximum of a set in $\mathcal{J}_{k, 3k+1}$ with $\min F = j$ is minimized 
when $F = F_j := [j, j+2k-2]_2$. 
Let $(x_i)_{1\le i\le k}$ be the length of the $k$ intervals in $F$, and let
$(y_j)_{1\le j\le k-1}$ be the $(k-1)$ gaps between consecutive intervals. 
We proceed by case analysis.

Case 1: $j\ge k+3$. Then $\max F\ge\max F_j \ge 3k+1$. Since sets in $\mathcal{J}_{k, 3k+1}$ can only contain
integers up to $3k+1$, the only set $F\in \mathcal{J}_{k, 3k+1}$ with $\min F\ge k+3$ is $F_{k+3}$.

Case 2: $j = k+2$. Then $F_{k+2} = [k+2, 3k]_2$. Since $\min F_{k+2} > |F_{k+2}|$, we may obtain other sets in two ways: first, we allow exactly one gap of $2$ between two integers in $F$, and second, we allow $F$ to contain two consecutive integers. For the former, we have
$$\sum_{i=1}^k x_i + \sum_{j=1}^{k-1} y_j \ =\ |\{k+2, k+3, \ldots, 3k+1\}|\ =\ 2k$$
with $x_1 = \cdots = x_k = 1$ and $y_j\ge 1$. Hence,
$\sum_{j=1}^{k-1} y_j  = k$, with $y_j\ge 1$.
By Lemma \ref{star}, the number of sets is $\binom{k-(k-1)+(k-2)}{k-2}  = k-1$.
For the latter, we have
$$\sum_{i=1}^{k} x_i + \sum_{j=1}^{k-1} y_j \ =\  |\{k+2, k+3, \ldots, 3k+1\}|\ =\ 2k$$
with $y_1 = \cdots = y_{k-1} = 1$ and $x_i\ge 1$. Hence,
$\sum_{i=1}^k x_i = k+1$, with $x_i\ge 1$.
By Lemma \ref{star}, the number of sets is $\binom{(k+1)-k+(k-1)}{k-1} = k$.
In this case, we have a total of $1+(k-1) + k = 2k$ sets. 

Case 3: $j = k+1$. Then $F_{k+1} = [k+1, 3k-1]_2$. Since $\min F_{k+1} >  |F_{k+1}|$, we can obtain other sets in the following ways.
\begin{itemize}
    \item We allow gaps between consecutive intervals to be up to $3$. Let $\ell$ be the sum of all the gaps. Then $k-1\le \ell\le k+1$. 
    We have
    $\sum_{j=1}^{k-1} y_j = \ell$, with $y_j\ge 1$.
    By Lemma \ref{star}, the number of sets is 
    \begin{align*}\sum_{\ell = k-1}^{k+1} \binom{\ell - (k-1)+ k-2}{k-2}&\ =\ \sum_{\ell = k-1}^{k+1} \binom{\ell-1}{k-2}\\
    &\ =\ \binom{k-2}{k-2} + \binom{k-1}{k-2} + \binom{k}{k-2}\\
    &\ =\ \binom{k+1}{2}.
    \end{align*}
    \item We allow $F$ to contain exactly $2$ consecutive integers and to either have exactly one gap of $2$ or not. Let $\ell$ be the sum of all the gaps of $F$. Then $k-1\le \ell\le k$.
    We have
    $$\sum_{i=1}^{k} x_i \ =\ k+1,\mbox{ with }x_i\ge 1\quad\mbox{ and }\quad \sum_{j=1}^{k-1} y_j \ =\ \ell, \mbox{ with }y_j\ge 1.$$
    By Lemma \ref{star}, the number of sets is
    $$\sum_{\ell = k-1}^{k}\binom{(k+1) - k + (k-1)}{k-1}\binom{\ell - (k-1)+(k-2)}{k-2}\ =\ k^2.$$
\end{itemize}
In this case, we have $\binom{k+1}{2} + k^2$ sets.

Case 4: $j = k$. Then $F_k = [k, 3k-2]_2$. Due to $\min F = k$, each interval in $F$ must have exactly $1$ element, but the gaps between consecutive elements in $F$ can reach $4$. Let $\ell$ be the sum of all the gaps. Then $k-1\le \ell\le k+2$.
We have
$\sum_{j=1}^{k-1} y_j = \ell$
with $y_j\ge 1$. 
By Lemma \ref{star}, the number of sets is
\begin{align*}
\sum_{\ell=k-1}^{k+2}\binom{\ell - (k-1) + k-2}{k-2} &\ =\ \sum_{\ell= k-1}^{k+2}\binom{\ell-1}{k-2}\\
&\ =\ \binom{k-2}{k-2} + \binom{k-1}{k-2} + \binom{k}{k-2} + \binom{k+1}{k-2}\\
&\ =\ 1 + (k-1) + \binom{k}{2} + \binom{k+1}{3}\\
&\ =\ \binom{k+1}{2} + \binom{k+1}{3}.
\end{align*}

Therefore, $|\mathcal{J}_{k, 3k}| = \binom{k+1}{3}+2\binom{k+1}{2} + (k+1)^2$,
as claimed. 
\end{proof}

\section{Relating $(|\mathcal{J}_{k,n}|)_{n=1}^\infty$ and $(|\mathcal{J}_{k+1,n}|)_{n=1}^\infty$}\label{woAI_appen}
To provide a rigorous combinatorial motivation for \eqref{AI_relation} rather than relying solely on
computational discovery, we offer a possible way of how to discover such a relation. A natural first reaction is to look at the difference sequences of each row of Table \ref{Data_oneint}. 
Given a sequence $(a_n)_{n=1}^\infty$, we use $d_i((a_n))$ to denote the $i$th difference sequence of $(a_n)_{n=1}^\infty$. In particular,
$$d_1((a_n))\ :=\  a_2-a_1, a_3-a_2, a_4-a_3, a_5-a_4, \ldots.$$
Supposing that $d_i((a_n))$ has been defined for some $i\in \mathbb{N}$, we let 
$$d_{i+1}((a_n)) \ :=\ d_1(d_i((a_n))).$$

We have
$$d_1((|\mathcal{J}_{1,n}|)) \ =\ 1, 2, 2, 3, 3, 4, 4, 5, 5, 6, 6, 7, 7, 8, 8, \ldots,$$
so we have the natural numbers, each of which is repeated twice. Meanwhile, 
\begin{align*}
&d_1((|\mathcal{J}_{2,n}|))\ =\  0, 0, 1, 2, 5, 8, 14, 20, 30, 40, 55, 70, 91, 112,\ldots\mbox{ and }\\
&d_2((|\mathcal{J}_{2,n}|))\ =\ 0, 1, 1, 3, 3, 6, 6, 10, 10, 15, 15, 21, 21, \ldots;
\end{align*}
hence, we need the second difference of $(|\mathcal{J}_{2,n}|)_{n=1}^\infty$ to see repeated differences. 
A similar phenomenon happens with $(|\mathcal{J}_{3,n}|)_{n=1}^\infty$, which requires us to take the third difference to see repetition:
\begin{align*}
&d_1((|\mathcal{J}_{3,n}|))\ =\  0, 0, 0, 0, 0, 1, 3, 9, 19, 39, 69, 119, 189, 294, 434, \ldots;\\
&d_2((|\mathcal{J}_{3,n}|))\ =\ 0, 0, 0, 0, 1, 2, 6, 10, 20, 30, 50, 70, 105, 140, \ldots;\mbox{ and }\\
&d_3((|\mathcal{J}_{3,n}|))\ =\ 0, 0, 0, 1, 1, 4, 4, 10, 10, 20, 20, 35, 35, \ldots.
\end{align*}

The above observation suggests that, instead of comparing $(|\mathcal{J}_{k,n}|)$ with $(|\mathcal{J}_{k+1, n}|)$ directly, we should compare $(|\mathcal{J}_{k,n}|)_{n=1}^\infty$ with $d_1((|\mathcal{J}_{k+1, n}|)_{n=1}^\infty)$. Table \ref{tab_adjust} juxtaposes $(|\mathcal{J}_{k,n}|)_{n=1}^\infty$ and $d_1((|\mathcal{J}_{k+1, n}|)_{n=1}^\infty)$.

\begin{table}[H]
\centering
\begin{tabular}{ |c| c| c| c| c| c| c| c| c| c| c| c| c| c| c| c| c| c| c| c| c|c|}
\hline
$n$ &$1$& $2$ & $3$ & $4$ & $5$ & $6$ & $7$ & $8$ & $9$ & $10$ & $11$ & $12$ & $13$ & $14$ & $15$ & $16$\\
\hline
$(|\mathcal{J}_{1,n}|)$ & $1$ & $2$ & $4$ & $6$ & $9$ & $12$ & $16$ & $20$ & $25$ & $30$ & $36$ & $42$ & $49$ & $56$ & $64$ & $72$  \\
\hline
$d_1((|\mathcal{J}_{2,n}|))$ &  & $0$ & $0$ & $1$ & $2$ &$5$ & $8$ & $14$ & $20$ & $30$ & $40$ & $55$ & $70$ & $91$ & $112$ & $140$  \\
\hline 
\hline
$|\mathcal{J}_{2,n}|$ & $0$ & $0$ & $0$ & $1$ & $3$ &$8$ & $16$ & $30$ & $50$ & $80$ & $120$ & $175$ & $245$ & $336$ & $448$ & $588$  \\
\hline 
$d_1((|\mathcal{J}_{3,n}|))$ &  & $0$ & $0$ & $0$ & $0$ & $0$ & $1$ & $3$ & $9$ & $19$ & $39$ & $69$ & $119$ & $189$ & $294$ & $434$ \\
\hline
\end{tabular}
\caption{A juxtaposition of $(|\mathcal{J}_{k,n}|)$ and $d_1((|\mathcal{J}_{k+1,n}|))$ with $k = 1,2$.}
\label{tab_adjust}
\end{table}

As expected, we see the following relationship between that the difference between  $(|\mathcal{J}_{k,n}|)_{n=1}^\infty$ and $d_1((|\mathcal{J}_{k+1, n}|)_{n=1}^\infty)$: the difference between a term  in 
$d_1((|\mathcal{J}_{k+1,n}|))$ and its second previous term is equal to a term in $(|\mathcal{J}_{k,n}|)$. For example,
in $d_1((|\mathcal{J}_{2,n}|))$,
$$91- 55 \ =\ 36, \quad 112- 70\ =\ 42, \quad \mbox{ and }\quad 140 - 91\ =\ 49.$$
In $d_1((|\mathcal{J}_{3,n}|))$, we have
$$39-9\ =\ 30, \quad 69-19\ =\ 50, \quad \mbox{ and }\quad 119-39\ =\ 80.$$
We, therefore, arrive at the relation
$$(|\mathcal{J}_{k+1,n+3}| - |\mathcal{J}_{k+1, n+2}|) - (|\mathcal{J}_{k+1, n+1}|-|\mathcal{J}_{k+1, n}|)\ =\ |\mathcal{J}_{k,n}|,\mbox{ for all }k,n\in\mathbb{N},$$
which is exactly \eqref{AI_relation}.

\ \\
\end{document}